\documentclass[11pt]{amsart}
\usepackage{amsmath}
\usepackage{amssymb}
\usepackage{amsthm}
\usepackage{amscd}
\usepackage[all]{xy}

\numberwithin{equation}{section}

\newcounter{AbcT}

\newtheorem {Theorem}    {Theorem}[section]

\newtheorem {Lemma}      [Theorem]    {Lemma}

\newtheorem {Proposition}[Theorem]    {Proposition}

\theoremstyle{remark}

\newtheorem {Definition} [Theorem]    {\bf{Definition}}
\def\acts{\curvearrowright}

\newcommand{\Cay}{\mathit{Cay}}

\newcounter{DM@bibnum}

\newcommand{\la}{\langle}
\newcommand{\ra}{\rangle}

\def\deg{{\rm deg\,}}

\def\NSL_2{{\mathcal N SL_2}}

\def\phi{\varphi}

\def\calT{{\mathcal T}}

\def\hbar{\bar h}

\newcommand{\cost}{\text{cost}}

\newcommand{\sub}{\subseteq}

\begin{document}

\title{Groups with Tarski number $5$}
\author{Gili Golan}
\address{Bar-Ilan University}
\email{gili.golan@math.biu.ac.il}


\begin{abstract}
The Tarski number of a non-amenable group $G$ is the minimal number of pieces
in a paradoxical decomposition of $G$. Until now the only numbers which were known to be Tarski numbers of some groups were $4$ and $6$. We construct a group with Tarski number $5$ and mention a related result for Tarski numbers of group actions.
\end{abstract}
\maketitle

\section{Introduction}
Recall the definition of a \emph{paradoxical decomposition} of a group.

\begin{Definition}\rm
\label{d1}

A group $G$ admits a \emph{paradoxical decomposition}
if there exist positive integers $m$ and $n$, disjoint subsets $P_1,\ldots, P_m,Q_1,\ldots, Q_n$ of $G$
and subsets $S_1=\{g_1,\ldots, g_m\}$, $S_2=\{h_1,\ldots, h_n\}$ of $G$ such that
\begin{equation}\label{eq1}G=\bigcup_{i=1}^m g_i P_i=\bigcup_{j=1}^n h_j Q_j.\end{equation}
The sets $S_1,S_2$ are called the \emph{translating sets} of the paradoxical decomposition.
\end{Definition}
It is well known \cite{Wa} that $G$ admits a paradoxical decomposition if and only if it is non-amenable.
The minimal possible value of $m+n$ in a paradoxical decomposition of $G$
is called the \emph{Tarski number} of $G$ and denoted by $\mathcal T(G)$.
 
It is clear that for any paradoxical decomposition we must have $m\geq 2$ and $n\geq 2$,
so the minimal possible value of Tarski number is $4$. By a theorem of J\'onsson and Dekker (see, for example, \cite[Theorem~5.8.38]{Sa}),  $\calT(G)=4$ if and only if $G$ contains a non-abelian free subgroup. 

Recently it was proved that the set of Tarski numbers is infinite \cite{OS}. At the time, no specific number other than $4$ was known to be a Tarski number. The first and only result of the kind appeared shortly afterwards in \cite{EGS}, where a group with Tarski number $6$ was constructed. The main feature of the proof was the use of random spanning forests on Cayley graphs. We shall use similar techniques to construct a group with Tarski number $5$. 

The notion of paradoxical decompositions and Tarski numbers naturally extends to group actions (see, for example, \cite{Wa}). In this more general setting the problem of determining whether a given number is a Tarski number can be completely resolved. Indeed, for every $n\ge 4$ it is possible to construct a faithful transitive group action with Tarski number $n$. The proof and related results about Tarski numbers of group actions will appear at a future paper 

\vskip .12cm
{\bf Acknowledgments.} The author would like to thank Mikhail Ershov and Andrei Jaikin Zapirain for useful discussions and Mark Sapir for useful discussions and comments on the text. Part of the research was done during the author's stay at the University of Virginia. She wishes to express her gratitude for the accommodations and hospitality.

\section{Groups with Tarski number $5$}

In what follows we shall make use of the following criterion which follows from Lemma 2.5 and Theorem 2.6 in \cite{EGS}.

\begin{Lemma}\label{lemma}
Let $G$ be a group and $S_1,S_2$ finite subsets of $G$. The following are equivalent.
\begin{enumerate}
\item $G$ has a paradoxical decomposition with translating sets $S_1,S_2$.
\item For any pair of finite subsets $A_1,A_2\in G$, $|A_1S_1\cup A_2S_2|\ge |A_1|+|A_2|.$
\end{enumerate}
\end{Lemma}

\begin{Proposition}
\label{prop:Betti}
Let $G$ be a group generated by $S=\{a,b,c\}$ and assume that $a$ is an element of infinite order. If $\cost(G)\ge 2.5$ then $G$ has a paradoxical decomposition with translating sets $S_1=\{1,a\},S_2=\{1,b,c\}$. 
\end{Proposition}

\begin{proof}
Fix a pair of finite subsets $A_1,A_2$ of $G$. By Lemma \ref{lemma} it suffices to prove that $|A_1S_1\cup A_2S_2|\ge |A_1|+|A_2|$.
Let $\Gamma=\Cay(G, S\cup S^{-1})$ be the right Cayley graph of $G$ with respect to $S\cup S^{-1}$ considered as an unoriented graph
without multiple edges. The key result we shall use is the theorem of Thom \cite[Theorem~3]{Thom} which asserts that there exists a $G$-invariant random spanning forest $\mu$ of $\Gamma$ such that $\mu$-a.s the forest contains all edges labeled by $a^{\pm 1}$ and the expected degree of a vertex in $\mu$ is at least $2 \cost(G\acts [0,1]^E)$ where $E$ is the set of edges of $\Gamma$ and the action $G\acts [0,1]^E$ is the natural action.

An immediate corollary is that there exists an ordinary
 forest $\mathcal F$ on $\Gamma$ (depending on $A_2$) such that all edges labeled by $a^{\pm 1}$ belong to $F$ and
\begin{equation}
\label{eq:forest}
\sum_{g\in A_2}\deg_{\mathcal F}(g)\geq 2 \cost(G\acts [0,1]^E)|A_2| \ge 2\cost(G)|A_2|\ge 5|A_2|.
\end{equation}
\vskip .12cm

Let $E$ be the set of all directed edges  $(g,gs)$ such that $g\in A_2$, $s\in S\cup S^{-1}$ and the unoriented edge
$\{g,gs\}$ lies in $\mathcal F$. Let $E_1$ be the subset of $E$ consisting of all edges $(g,gs)\in E$ with $s\in S\setminus S^{-1}$. Note that $|E|\geq 5|A_2|$ by \eqref{eq:forest}, and it is clear that $|E_1|\geq |E|-|S||A_2|$,
so that $|E_1|\geq 2|A_2|$.

Since the sets $S\setminus S^{-1}$ and $(S\setminus S^{-1})^{-1}$ are disjoint, $E_1$ does not contain
a pair of opposite edges. Also, the label of every edge in $E_1$ belongs to $\{a,b,c\}$. 
Thus, if $E_2$ denotes the set of edges $(g,gs)\in E_1$ such that $g\in A_2$ and $s\in \{b,c\}$, $|E_2|\ge  |E_1|-|A_2|\ge |A_2|$. 

Let $E_3$ be the set of directed edges $(g,ga)$ for $g\in A_1$. Clearly, $E_2$ and $E_3$ are disjoint sets and $E_2\cup E_3$ does not contain a pair of opposite edges.
The endpoints of edges in $E_2\cup E_3$ lie in the set $A_1S_1\cup A_2S_2$.
Let $\Lambda$ be the unoriented graph with vertex set $A_1S_1\cup A_2S_2$
and edge set $E_2\cup E_3$ (with forgotten orientation). Then $\Lambda$ is a subgraph of $\mathcal F$;
in particular $\Lambda$ is a (finite) forest. Hence
$$|A_1S_1\cup A_2S_2|=|V(\Lambda)|>|E(\Lambda)|= |E_3|+|E_2|\ge |A_1|+|A_2|,$$ as desired.
\end{proof}

\begin{Theorem}\label{Cor}
Let $F=\la a,b,c\ra$ be a free group of rank $3$. Let $r_1,r_2,\dots\in \gamma_2F$ be an enumeration of the elements of the derived subgroup of $F$. Let $R=\{r_i^{p^{n_i}}\}$ for some integer sequence $n_1,n_2,\ldots$ such that $\sum_{i}\frac{1}{p^{n_i}}\le\frac{1}{2}$. Then for $G=\la X| R\ra$, $\mathcal T(G)=5$.
\end{Theorem}

\begin{proof}
By \cite[Theorem B.1]{EGS} $G$ has a quotient $Q$ such that $\beta_1(Q)\ge 1.5$ where $\beta_1(Q)$ is the first $L^2$-Betti number of $Q$. In addition, $Q$ can be chosen so that the image of $a$ in $Q$ has infinite order. Indeed, as in the proof of \cite[Theorem B.1]{EGS}, let $R_m=\{r_i^{p^{n_i}}\}_{i=1}^m$, $G(m)=\la X|R_m\ra$ and $G(m)_p$ be the image of $G(m)$ in its pro-$p$ completion. Then, if $G(m)_p=F/N_m$, by the argument in \cite{EGS}, for $Q=F/\bigcup_{m\in\mathbb{N}}N_m$, $\beta_1(Q)\ge 1.5$. Assume by contradiction that the image of $a$ in $Q$ has finite order. Then, for some $m\in\mathbb{N}$, the image of $a$ in $G(m)_p$ is also of finite order. Let $Ab\colon F\to \mathbb{Z}^3$ be the abelianization homomorphism. Since $R_m\sub \gamma_2F$, $Ab$ induces a homomorphism $Ab\colon G(m)\to \mathbb{Z}^3$. Let $i_{\mathbb{Z}^3}\colon \mathbb{Z}^3\to\widehat{\mathbb{Z}^3}_p$ and $i_{G(m)}\colon G(m)\to \widehat {G(m)}_p$ be the natural homomorphisms from $\mathbb{Z}^3$ and $G(m)$ to their pro-$p$ completions. Then, $\phi=i_{\mathbb{Z}^3}\circ Ab\colon G(m)\to \widehat{Z^3}_p$ is a homomorphism from $G(m)$ to a pro-$p$ group. Clearly, $\phi$ is continuous when $G(m)$ is equipped with the pro-$p$ topology. Thus, it can be extended in a unique way to a homomorphism $\psi \colon \widehat{G(m)}_p\to \widehat{Z^3}_p$ for which $\psi \circ i_{G(m)}=\phi$. Since the image of $a$ under $\phi$ has infinite order in $\widehat{Z^3}_p$, $i_{G(m)}(a)$ must have an infinite order in $G(m)_p\sub\widehat{G(m)}_p$, a contradiction.  

Now we are ready to prove the theorem. For a quotient $Q$ with the properties mentioned above, $\cost(Q)\ge \beta_1(Q)+1\ge 2.5$. Thus, Proposition \ref{prop:Betti} applied to $Q$ implies that $\mathcal T(Q)\le 5$. Since $Q$ is a quotient of $G$, $\mathcal T(G)\le \mathcal T(Q)\le 5$ (see, for example, \cite[Theorem 5.8.16]{Sa}). Since $G$ is torsion-by-abelian it doesn't contain any free non abelian subgroup. Hence, by the theorem of J\'onsson and Dekker mentioned above $\mathcal T(G)\neq 4$. Therefore $\mathcal T(G)=5$ as required. 
\end{proof}

\end{document}